\documentclass[12pt]{amsart}

\usepackage{amsmath}
\usepackage{amssymb}
\usepackage{amsthm}
\usepackage{amscd}
\usepackage{delarray}
\usepackage{hyperref}
\headsep=1cm \numberwithin{equation}{section}
\def\F{{\mathbb F}}

\def\P{{\mathbb P}}

\def\Z{{\mathbb Z}}

\newtheorem{theorem}{Theorem}[section]
\newtheorem{lemma}[theorem]{Lemma}

\newtheorem{proposition}[theorem]{Proposition}
\newtheorem{corollary}[theorem]{Corollary}

\theoremstyle{definition}

\newtheorem{remark}[theorem]{Remark}

\newtheorem{convention and reminder}[theorem]{Convention and Reminder}
\newtheorem{convention and remark}[theorem]{Convention and Remark}
\newtheorem{definition and remark}[theorem]{Definition and Remark}

\newtheorem{reminders and definition}[theorem]{Reminders and Definition}

\newtheorem{notation and remarks}[theorem]{Notation and Remarks}
\newtheorem{notation and remark}[theorem]{Notation and Remark}

\newcommand\reg{\operatorname{reg}}
\newcommand\Ker{\operatorname{\Ker}}

\begin{document}

\title{On hyperquadrics containing projective varieties}

\author{Euisung Park}
 
\address{Korea University, Department of Mathematics, Anam-dong, Seongbuk-gu, Seoul 136-701, Republic of Korea}
\email{euisungpark@korea.ac.kr}


\begin{abstract}
Classical Castelnuovo Lemma shows that the number of linearly independent quadratic equations of a nondegenerate irreducible projective variety of codimension $c$ is at most ${{c+1} \choose {2}}$ and the equality is attained if and only if the variety is of minimal degree. Also G. Fano's generalization of Castelnuovo Lemma implies that 
the next case occurs if and only if the variety is a del Pezzo variety. Recently, these results are extended to the next case in \cite{Pa3}. This paper is intended to complete the classification of varieties satisfying at least ${{c+1} \choose {2}}-3$ linearly independent quadratic equations. Also we investigate the zero set of those quadratic equations and apply our results to projective varieties of degree $\geq 2c+1$.
\end{abstract}

\maketitle
\thispagestyle{empty}

\section{Introduction}
\noindent Throughout this paper, we work over an algebraically
closed field $\Bbbk$ of characteristic zero. For a closed subscheme
$Z \subset \P^r$, we denote by $a_2 (Z)$ the number of linearly
independent quadratic equations of $Z$. That is, $a_2 (Z) = h^0
(\P^r , \mathcal{I}_Z (2))$.

Let $X \subset \P^{n+c}$ be a nondegenerate irreducible
projective variety of dimension $n$ and degree $d$. As it was
indicated in \cite{L}, the natural problem of finding an upper bound
of $a_2 (X)$ and classifying the borderline cases can be answered by
the classical Castelnuovo Lemma and its generalization by G. Fano (cf. \cite{C} and
\cite{Fa}).

\begin{theorem}[Theorem 1.2 and Theorem 1.6 in \cite{L}]\label{thm:known results 1 classical}
Let $X \subset \P^{n+c}$ be as above. Then
\smallskip

\renewcommand{\descriptionlabel}[1]%
             {\hspace{\labelsep}\textrm{#1}}
\begin{description}
\setlength{\labelwidth}{13mm} \setlength{\labelsep}{1.5mm}
\setlength{\itemindent}{0mm}

\item[$(1)$] ${\rm (G.~Castelnuovo)}$ $a_2 (X) \leq {{c+1} \choose {2}}$; the equality is attained if and only if $d=c+1$.
\item[$(2)$] ${\rm (G.~Fano)}$ If $c \geq 2$, then $a_2 (X) = {{c+1} \choose {2}}-1$ if and only if
a general linear curve section of $X$ is a linearly normal curve of
arithmetic genus one.
\end{description}
\end{theorem}

Varieties appearing in Theorem \ref{thm:known results 1
classical}.(2) are nowadays called \textit{del Pezzo varieties}, due
to T. Fujita (cf. \cite[Chapter I. $\S~6$]{Fu3}). They are exactly
\textit{varieties of almost minimal degree} which are arithmetically
Cohen-Macaulay (cf. \cite[Theorem 6.2]{BS2}).

In \cite{Pa3}, the author approached this problem through a more
basic method of ``hyperplane section argument". Indeed, let $\Gamma
\subset \P^c$ be a general zero dimensional linear section of $X$.
Thus it is a finite set of $d$ points in general position. Also let
$\Gamma_0$ be a subset of $\Gamma$ such that $|\Gamma_0 | =
\mbox{min} \{d, 2c+1 \}$. Then $\Gamma_0 \subset \P^c$ is $2$-normal
and hence
\begin{equation*}
a_2 (\Gamma_0 ) = {{c+2} \choose {2}}- |\Gamma_0|.
\end{equation*}
Also it holds that $a_2 (X) \leq a_2 (\Gamma) \leq a_2 (\Gamma_0)$.
These elementary observations give us the following useful
inequality:
\begin{equation}\label{eq:fundamental inequality}
a_2 (X)   \leq {{c+1} \choose {2}}-\mbox{min} \{d-c-1,c \}
\end{equation}
(cf. \cite[Proposition 4.2 and Corollary 4.3]{Pa3}). For example, it
implies that the value of $a_2 (X)$ is at most ${{c+1} \choose {2}}$ and equality can
occur only when $d=c+1$. Also it shows that if $a_2 (X) = {{c+1}
\choose {2}}-1$ then $d \leq c+2$. In \cite[Theorem 4.5]{Pa3},
Theorem \ref{thm:known results 1 classical} was reproved and
generalized to the next case by combining (\ref{eq:fundamental
inequality}) and the structure theory of ``low degree varieties"
(eg. \cite{HSV}, \cite{BS1}, \cite{BS2}, etc).

\begin{theorem}[Theorem 1.3.(4) in \cite{Pa3}]\label{thm:known results 2
Park2015}
Let $X \subset \P^{n+c}$ be as above. If $c \geq 3$, then $a_2 (X) =
{{c+1} \choose {2}}-2$ if and only if either $d=c+2$ and ${\rm
depth}(X)=n$ or else $d=c+3$ and ${\rm depth}(X)=n+1$.
\end{theorem}

Along this line, this paper is intended to study the following three problems induced from (\ref{eq:fundamental inequality}):\\
\begin{enumerate}
\item[(1)] Classify all $X \subset \P^{n+c}$ satisfying the condition $a_2 (X) \leq {{c+1} \choose
{2}}-3$.
\item[(2)] By (\ref{eq:fundamental inequality}), if $d \geq 2c+1$ then the value of $a_2 (X)$ is at most ${{c} \choose {2}}$. Then classify all $X \subset \P^{n+c}$ such that $d \geq 2c+1$ and $a_2
(X) \leq {{c} \choose {2}} - 3$.
\item[(3)] For a projective variety $X \subset \P^{n+c}$, we
denote by $\textsf{Q}(X)$ the zero set of $I(X)_2$ in $\P^{n+c}$.
Namely, $\textsf{Q}(X) := \mbox{Bs}(|\mathcal{I}_X (2)|)$. Then investigate the irreducible decomposition of $\textsf{Q}(X)$ when $X \subset \P^{n+c}$ appears as an answer for the previous two problems (1) and
(2).\\
\end{enumerate}

Our first main result in this paper is the generalization of Theorem
\ref{thm:known results 2 Park2015} to the next case and hence it
gives an answer for Problem (1). The \textit{sectional genus} of
$X$, denoted by $g(X)$, is defined to be the arithmetic genus of a
general linear curve section of $X$.

\begin{theorem}\label{first main theorem}
Let $X \subset \P^{n+c}$ be as above. If $c \geq 4$, then $a_2 (X) =
{{c+1} \choose {2}}-3$ if and only if $X$ satisfies one of the
following conditions:
\begin{enumerate}
\item[$(i)$] $d=c+2$ and ${\rm depth}(X)=n-1$ $($and hence $n \geq 2$$)$;
\item[$(ii)$] $d=c+3$ and $X$ is a variety of maximal sectional
regularity;
\item[$(iii)$] $d=c+3$, $g(X)=1$ and ${\rm depth}(X)=n$;
\item[$(iv)$] $d=c+4$ and ${\rm depth}(X)=n+1$.
\end{enumerate}
\end{theorem}

The proof of this result is provided in $\S~3$.

See $\S~2$ where we review basic properties of varieties of
almost minimal degree and varieties of maximal sectional regularity
(including their definition).

Next, we consider the inequality (\ref{eq:fundamental inequality})
for the cases $d \geq 2c+1$. Then $a_2 (X)$ is at most ${{c} \choose
{2}}$ and it takes the possible maximal value ${{c} \choose {2}}$ if
$X$ is contained in a variety of minimal degree as a divisor. It is
one of the most beautiful applications of the classical Castelnuovo
Lemma that the converse is also true if $d \geq 2c+3$. Also Fano's
generalization of Castelnuovo Lemma implies that if $d
\geq 2c+5$ then $a_2 (X) = {{c} \choose {2}}-1$ if and only if $X$
is contained in an $(n+1)$-dimensional del Pezzo variety. In this
direction, we obtain the following our second main result.

\begin{theorem}\label{second main theorem}
Let $k$ be an integer in $\{ 0,1,2,3 \}$. Suppose that $c \geq k+2$
if $0 \leq k \leq 2$ and $c \geq 6$ if $k=3$. Let $X \subset
\P^{n+c}$ be a nondegenerate irreducible projective variety of
dimension $n$ and degree $d \geq 2c+2k+3$. Then the following two
conditions are equivalent:
\begin{enumerate}
\item[$(i)$] $a_2 (X) = {{c} \choose {2}}-k$.
\item[$(ii)$] There exists a unique $(n+1)$-dimensional variety $Y
\subset \P^{n+c}$ such that $X \subset Y$ and $a_2 (Y) = {{c}
\choose {2}}-k$.
\end{enumerate}
\end{theorem}

The proof of this result is provided in $\S~5$. See Theorem \ref{thm:irreducible decomposition 2}.

In the proof of Theorem \ref{second main theorem}, the existence of
$Y$ is shown by using the Castelnuovo theory, including I.
Petrakiev's recent result in \cite{Pe}. See Theorem
\ref{thm:Castelnuovo Theory}.

Once we know the existence of an $(n+1)$-dimensional variety $Y$, it
follows by Bezout's theorem that $Y$ is an irreducible component of
$\textsf{Q}(X) = \mbox{Bs}(|\mathcal{I}_X (2)|)$. So, to finish the
proof of Theorem \ref{second main theorem} by showing the uniqueness
of $Y$, we investigate $\textsf{Q}(X)$ when $X$ is one of the
varieties listed in Theorem \ref{thm:known results 1 classical},
Theorem \ref{thm:known results 2 Park2015} and Theorem \ref{first main theorem}.
Some of them are cut out by quadrics and hence $\textsf{Q}(X)$ is
$X$ itself. On the other hand, $\textsf{Q}(X)$ can be strictly
bigger than $X$ for some $X$ (eg. see Proposition \ref{prop:VMSR
d=c+3}.(3) and Remark \ref{rmk:trivial cases}). Our third main
result is about a common property of $\textsf{Q}(X)$ for all those $X$'s.

\begin{theorem}\label{thm:irreducible decomposition 1}
Suppose that $k \in \{0,1,2,3\}$ and $c \geq k+1$. If $X \subset
\P^{n+c}$ is an $n$-dimensional nondegenerate irreducible projective
variety such that
\begin{equation*}
a_2 (X) = {{c+1} \choose {2}}-k,
\end{equation*}
then $X$ is the only nondegenerate irreducible component of
$\textsf{Q}(X)$ of dimension $\geq n$.
\end{theorem}

The proof of this result is provided in $\S~4$.

The organization of this paper is as follows. In $\S 2$, we review
basic properties of varieties of almost minimal degree and varieties
of maximal sectional regularity. In $\S 3$, we give a proof of
Theorem \ref{first main theorem}. Finally, $\S 4$ and $\S 5$ are
devoted to investigating $\textsf{Q}(X)$ when $X$ is one of the
varieties listed in Theorem \ref{thm:known results 1 classical},
\ref{thm:known results 2 Park2015} and \ref{first main theorem}.\\

\noindent {\bf Acknowledgement.} This work was supported by the
Korea Research Foundation Grant funded by the Korean Government (no.
2018R1D1A1B07041336).

\section{Some projective varieties having large $a_2 (X)$}
\noindent This section is devoted to reviewing some known results
about varieties of almost minimal degree and varieties of maximal
sectional regularity since they appear in our main results in the
present paper. Throughout this section, let
\begin{equation*}
X \subset \P^{n+c}
\end{equation*}
be a nondegenerate projective variety of dimension $n$ and degree
$d$.

\subsection{Varieties of almost minimal degree} Due to \cite{BS2}, we say that $X \subset \P^{n+c}$ is a \textit{variety of almost
minimal degree} if $d = c+2$. The results of \cite{Fu1} and
\cite{Fu2} imply that these varieties can be divided into two
classes:
\begin{enumerate}
\item[$1.$] $X \subset \P^{n+c}$ is linearly normal and $X$ is normal;
\item[$2.$] $X \subset \P^{n+c}$ is non-linearly normal or $X$ is non-normal.
\end{enumerate}
By \cite[Theorem 2.1 b)]{Fu1}, see also $(6.4.6)$ and $(9.2)$ in
\cite{Fu}, $X$ is of the first type if and only if it is a normal
del Pezzo variety. By \cite[Theorem 2.1 a)]{Fu1}, $X$ is of the
second type if and only if
\begin{equation*}
X = \pi_p (\widetilde{X})
\end{equation*}
where $\widetilde{X} \subset \P^{n+c+1}$ is an $n$-dimensional
variety of minimal degree and $\pi_p : \widetilde{X} \rightarrow
{\mathbb P}^r$ is the linear projection of $\widetilde{X}$ from a
closed point $p$ in ${\mathbb P}^{n+c+1} \setminus \widetilde{X}$.
In this case, one can naturally expect that all the properties of
$X$ may be precisely described in terms of the relative location of
$p$ with respect to $\widetilde{X}$. For the cohomological and local
properties of $X$, this expectation turns out to be true in
\cite{BS2}. Indeed those properties are governed by the
\textit{secant locus} $\Sigma_p (\widetilde{X})$ of $\widetilde{X}$
with respect to $p$, which is defined to be the scheme-theoretic
intersection of $\widetilde{X}$ and the union of all secant lines to
$X$ passing through $p$. In \cite{BP1}, the authors obtained a
classification theory of the second type by describing the so-called
\textit{secant stratification} of $\widetilde{X} \subset \P^{n+c+1}$ in
terms of the secant locus $\Sigma_p (\widetilde{X})$.

\begin{theorem}\label{thm:VAMD review}
Suppose that $c \geq 3$ and let $X \subset \P^{n+c}$ be such that $X
= \pi_p (\widetilde{X})$ where $\widetilde{X} \subset {\mathbb
P}^{n+c+1}$ is an $n$-dimensional variety of minimal degree and $p$
is a closed point in ${\mathbb P}^{n+c+1} \setminus \widetilde{X}$.
Then
\smallskip

\renewcommand{\descriptionlabel}[1]%
             {\hspace{\labelsep}\textrm{#1}}
\begin{description}
\setlength{\labelwidth}{13mm} \setlength{\labelsep}{1.5mm}
\setlength{\itemindent}{0mm}

\item[$(1)$] $($\cite[Theorem A and B]{HSV}$)$ $a_2 (X) = {{c+1} \choose {2}} + \mbox{depth}(X)-n-2$.

\item[$(2)$] $($\cite[Theorem 1.3]{BS2}$)$ $\mbox{depth}(X) = \mbox{dim}~\Sigma_p (\widetilde{X})+2 =
\mbox{dim}~ \mbox{Sing}(X) +2$.

\item[$(3)$] $($\cite[Theorem 1.4]{BS2} and \cite[Theorem 5.5 and 5.8]{BP2}$)$ $X$ is contained in an
$(n+1)$-dimensional rational normal scroll $Y$ such that
$\mbox{Sing}(X)=\mbox{Vert}(Y)$.
\smallskip

\item[$(4)$] $($\cite[Proposition 3.2]{BP1}$)$ $\mbox{dim}~\mbox{Vert}(\tilde{X}) \leq  \mbox{dim}~\Sigma_p
(\widetilde{X}) \leq \mbox{dim}~\mbox{Vert}(\tilde{X}) +3$.
\end{description}
\end{theorem}

Note that the cases where $\mbox{depth}(X)=n$ and $n-1$ appear in Theorem
\ref{thm:known results 2 Park2015} and Theorem \ref{first main
theorem}. In the following corollary, we show that if $X$ is not a
cone then the dimension of $X$ has an upper bound.

\begin{corollary}\label{cor:depth n and n-1}
Let $X \subset \P^{n+c}$ be as in Theorem \ref{thm:VAMD review}.
Suppose that $X$ is not a cone $($or, equivalently, $\tilde{X}$ is
not a cone$)$. Then
\smallskip

\renewcommand{\descriptionlabel}[1]%
             {\hspace{\labelsep}\textrm{#1}}
\begin{description}
\setlength{\labelwidth}{13mm} \setlength{\labelsep}{1.5mm}
\setlength{\itemindent}{0mm}

\item[$(1)$] If $\mbox{depth}(X)=n$, then $n \leq 4$.

\item[$(2)$] If $n \geq 2$ and $\mbox{depth}(X)=n-1$, then $n \leq 5$.
\end{description}
\end{corollary}

\begin{proof}
Since $\tilde{X}$ is smooth, we have
$\mbox{dim}~\mbox{Vert}(\tilde{X})=-1$. Thus we have
\begin{equation}\label{eq:depth formula}
-1 \leq \mbox{dim}~\Sigma_p (\widetilde{X}) = \mbox{depth}(X)-2 \leq
2
\end{equation}
by Theorem \ref{thm:VAMD review}.$(2)$ and $(4)$.
\smallskip

\noindent $(1)$ By (\ref{eq:depth formula}), we have $n \leq 4$ if
$\mbox{depth}(X)=n$.
\smallskip

\noindent $(2)$ By (\ref{eq:depth formula}), we have $n \leq 5$ if
$\mbox{depth}(X)=n-1$.
\end{proof}

\subsection{Varieties of maximal sectional regularity} Let
\begin{equation*}
\mathcal{C} \subset \P^{1+c}
\end{equation*}
be a general linear curve section of $X$. By \cite{GLP}, the
Castelnuovo-Mumford regularity $\reg (\mathcal{C})$ is at most
$d-c+1$. Due to \cite{BLPS1}, we say that $X \subset \P^{n+c}$ is a
\textit{variety of maximal sectional regularity} if $d \geq c + 3$
and $\reg (\mathcal{C})$ is equal to the maximal possible value $d
-c+1$. By the classification result in \cite{BLPS1}, if $c \geq 4$ then $X$
is a variety of maximal sectional regularity if and only if there is
an $n$-dimensional linear subspace $\mathbb{F}(X) \subset \P^r$ such
that $X \cap \mathbb{F}(X) \subset \mathbb{F}(X)$ is a hypersurface
of degree $d-c+1$. Note that this hypersurface gives a $(d-c+1)$-secant line to $\mathcal{C}$ for all general linear curve section $\mathcal{C}$ of $X$.

Due to Theorem \ref{first main theorem}, we are interested in the
case where $d=c+3$.

\begin{proposition}\label{prop:VMSR d=c+3}
Suppose that $c \geq 4$ and let $X \subset \P^{n+c}$ be an
$n$-dimensional variety of maximal sectional regularity and of
degree $d=c+3$. Then
\smallskip

\renewcommand{\descriptionlabel}[1]%
             {\hspace{\labelsep}\textrm{#1}}
\begin{description}
\setlength{\labelwidth}{13mm} \setlength{\labelsep}{1.5mm}
\setlength{\itemindent}{0mm}

\item[$(1)$] $\mbox{depth}(X)=n$ and $\mbox{depth}(X \cup \F
(X))=n+1$.
\item[$(2)$] $X \cup \F(X)$ is ideal theoretically cut out by
quadrics.
\item[$(3)$] $I(X)_2 = I(X \cup \F(X))_2$ and hence $\textsf{Q}(X) = X \cup \F(X)$.
\end{description}
\end{proposition}

\begin{proof}
$(1)$ For $n=1$ and $n=2$, see respectively \cite[Proposition
3.5]{BS1} and \cite[Theorem 4.1]{BLPS2}. For $n \geq 3$, let $S = X
\cap \P^{2+c}$ be a general linear surface section of $X$. Then $S$
is a surface of maximal sectional regularity and $\F(S) = \F(X) \cap
\P^{2+c}$. In particular,
\begin{equation*}
(X \cup \F(X)) \cap \P^{2+c} = S \cup \F (S).
\end{equation*}
Therefore it follows that
\begin{equation*}
\mbox{depth}(X) = \mbox{depth} (S)+(n-2) =n
\end{equation*}
and
\begin{equation*}
\mbox{depth}(X \cup \F (X)) = \mbox{depth} (S\cup \F (S))+(n-2)
=n+1.
\end{equation*}

\noindent $(2)$ Let $\Gamma \subset \P^c$ be a general zero
dimensional linear section of $X \cup \F (X)$. Then $\Gamma$ is a
finite set of $c+4$ points in general position by \cite[Theorem
1.1]{BS3}. Since $c \geq 4$, we have $|\Gamma| \leq 2c$ and hence
$I(\Gamma)$ is generated by quadrics (cf. \cite[Theorem 1]{GL}).
This shows that $X \cup \F(X)$ is ideal theoretically cut out by
quadrics since $\mbox{depth}(X \cup \F (X))=n+1$.

\noindent $(3)$ By $(1)$ and \cite[Proposition 4.3]{BLPS2}, we have
\begin{equation*}
a_2 (X) = a_2 (S) = a_2 (S \cup \F (S)) = a_2 (X \cup \F (S)).
\end{equation*}
Obviously, this shows that $I(X)_2$ is equal to $I(X \cup \F(X))_2$.
\end{proof}

\section{Proof of Theorem \ref{first main theorem}}
\noindent This section is devoted to proving Theorem \ref{first main
theorem}. We begin with two useful known facts.

\begin{proposition}[Corollary 4.3 and 4.4 in \cite{Pa3}]\label{prop:degree bound}
Let $c$ and $k$ be integers with $1 \leq k \leq c$ and let $X \subset \P^{n+c}$ be a nondegenerate projective variety of dimension $n$ and degree $d$.
\smallskip

\renewcommand{\descriptionlabel}[1]%
             {\hspace{\labelsep}\textrm{#1}}
\begin{description}
\setlength{\labelwidth}{13mm} \setlength{\labelsep}{1.5mm}
\setlength{\itemindent}{0mm}

\item[$(1)$] If $a_2 (X) = {{c+1} \choose
{2}}+ 1-k$, then $d \leq c+k$.

\item[$(2)$] If $d=c+k$, then
\begin{equation*}
a_2 (X) = {{c+1} \choose {2}}+ 1-k \quad \Leftrightarrow \quad
\mbox{depth}(X)=n+1.
\end{equation*}
\end{description}
\end{proposition}

\begin{corollary}\label{cor:degree range}
Let $X \subset \P^{n+c}$ be a nondegenerate projective variety of
codimension $c \geq 4$ and degree $d$. If $a_2 (X) = {{c+1} \choose
{2}}-3$, then $c+2 \leq d \leq c+4$. 
\end{corollary}

\begin{proof}
By Proposition \ref{prop:degree bound}.(1), it holds that $c+1 \leq d \leq c+4$. Also $d \neq c+1$ since if $d=c+1$, then $a_2 (X) = {{c+1} \choose {2}}$ by Theorem \ref{thm:known results 1 classical}.
\end{proof}

\begin{proposition}[Theorem 3.11.(4) in \cite{Pa3}]\label{prop:curve
case} Suppose that $c \geq 4$ and let $\mathcal{C} \subset \P^{c+1}$ be a nondegenerate projective integral curve of degree $d$. Then $a_2 (\mathcal{C}) = {{c+1} \choose {2}} -3$ if and
only if is either
    \begin{enumerate}
    \item[(i)] $d=c+3$ and $\mathcal{C}$ is a smooth rational curve having a $4$-secant line
    \end{enumerate}
    or
    \begin{enumerate}
    \item[(ii)] $d=c+3$ and $\mathcal{C}$ is the image of an isomorphic projection of a linearly normal curve of arithmetic genus $1$ from a point
    \end{enumerate}
    or else
    \begin{enumerate}
    \item[(iii)] $d=c+4$ and $\mathcal{C}$ is a linearly normal curve of arithmetic genus $3$.
    \end{enumerate}
\end{proposition}

Now, we want to investigate a few common properties of nondegenerate
projective varieties $X \subset \P^{n+c}$ of codimension $c \geq 4$
such that $a_2 (X) = {{c+1} \choose {2}}-3$.

\begin{lemma}\label{lem:hyperplane}
Let $X \subset \P^{n+c}$ be a nondegenerate irreducible projective
variety of dimension $n \geq 2$, codimension $c \geq 4$ and degree
$d \geq c+3$ and let $Y \subset \P^{n-1+c}$ be a general hyperplane
section of $X$. If $a_2 (X) = {{c+1} \choose {2}}-3$, then $a_2 (Y)
= {{c+1} \choose {2}}-3$.
\end{lemma}

\begin{proof}
Obviously, it holds that $a_2 (Y) \geq a_2 (X)$ and equality is
attained if ${\rm depth}(X) \geq 2$. So, we need to consider the
case where ${\rm depth}(X) = 1$ and hence ${\rm depth}(Y) = 1$. In
this case, if $a_2 (Y) \geq {{c+1} \choose {2}}-2$ then $d=c+3$ and
${\rm depth}(Y) = n \geq 2$ by Theorem \ref{thm:known results 1
classical} and \ref{thm:known results 2 Park2015}. Therefore it must
hold that $a_2 (Y) = {{c+1} \choose {2}}-3$.
\end{proof}

\begin{proposition}\label{prop:depth for d=c+3}
Let $X \subset \P^{c+n}$ be a nondegenerate irreducible projective
variety of codimension $c \geq 4$ and degree $d = c+3$ such that
$a_2 (X) = {{c+1} \choose {2}}-3$. Then ${\rm depth}(X) = n$.
\end{proposition}

\begin{proof}
If $n=1$, then the assertion comes from Proposition \ref{prop:curve
case}. Suppose that the statement of this lemma holds for $n=2$.
Then it holds for all $n \geq 2$ by Lemma \ref{lem:hyperplane}. So
we concentrate on proving the case of $n=2$.

From now on, suppose that $n=2$. Let $\mathcal{C} \subset \P^{c+1}$
be a general hyperplane section of $X$. Lemma \ref{lem:hyperplane}
shows that $a_2 (\mathcal{C}) = {{c+1} \choose {2}}-3$. According to
Proposition \ref{prop:curve case}, there are two cases.

First, suppose that $\mathcal{C}$ is a smooth rational curve of
degree $d=c+3$ having a $4$-secant line. Thus $\mathcal{C}$ is a
curve of maximal regularity and hence $X$ is a \textit{surface of
maximal sectional regularity} due to \cite{BLPS1}. Note that
\begin{equation*}
d=c+3 \leq 2(2+c)-4
\end{equation*}
since $c \geq 4$. Thus we get ${\rm depth}(X) = 2$ by \cite[Theorem
4.1]{BLPS2}.

Next, assume that $\mathcal{C}$ is the image of an isomorphic
projection of a linearly normal curve $\tilde{\mathcal{C}} \subset
\P^{2+c}$ of arithmetic genus $1$ from a point. Then
$\tilde{\mathcal{C}}$ satisfies Green-Lazarsfeld's condition $N_{c-1}$ and hence
$\mathcal{C}$ is $3$-regular (cf. \cite[Theorem 3.2]{BLPS2}). We
first show that $X$ is linearly normal. Indeed, if not
then it is the image of an isomorphic projection of a surface
$\tilde{X} \subset \P^{c+3}$ from a point. Note that ${\rm
deg}(\tilde{X})=c+3$ and a general hyperplane section of $\tilde{X}$
is of arithmetic genus $1$. Namely, $\tilde{X}$ is a del Pezzo
surface (cf. \cite[Chapter I. $\S~6$]{Fu3}). Then $X$ is $3$-regular (cf.
\cite[Corollary 3.3]{BLPS2}). Now, consider the long exact sequence
\begin{equation*}
0 \rightarrow H^0 (\P^{c+2} , \mathcal{I}_X (2)) \rightarrow H^0
(\P^{c+1} , \mathcal{I}_{\mathcal{C}} (2)) \rightarrow H^1 (\P^{c+2}
, \mathcal{I}_X (1)) \rightarrow H^1 (\P^{c+2} , \mathcal{I}_X (2))
=0
\end{equation*}
Since $a_2 (X) = a_2 (\mathcal{C})$ by Lemma
\ref{lem:hyperplane}, it follows that $H^1 (\P^{c+2} , \mathcal{I}_X (1)) =0$, 
which contradicts to our assumption that $X$ is not linearly normal.
In consequence, it is shown that $X$ is linearly normal.

To show that ${\rm depth}(X) \geq 2$, we use the long exact sequence
\begin{equation}\label{eq:higher normality}
\cdots \rightarrow H^1 (\P^{c+2} , \mathcal{I}_X (j)) \rightarrow
H^1 (\P^{c+2} , \mathcal{I}_X (j+1)) \rightarrow H^1 (\P^{c+1} ,
\mathcal{I}_{\mathcal{C}} (j+1)) \rightarrow \cdots.
\end{equation}
Note that the third term vanishes for all $j \geq 1$ since $\mathcal{C}$ is
$3$-regular. For $j=1$, we get
\begin{equation*}
H^1 (\P^{c+2} , \mathcal{I}_X (2))=0
\end{equation*}
since $X$ is linearly normal. Repeating this argument, we can see
the vanishing
\begin{equation*}
H^1 (\P^{c+2} , \mathcal{I}_X (j))=0
\end{equation*}
for all $j \geq 1$. This completes the proof that ${\rm depth}(X)
\geq 2$.

Finally, Theorem \ref{thm:known results 2 Park2015} shows that if
${\rm depth}(X) = 3$ then $a_2 (X) = {{c+1} \choose {2}}-2$.
Therefore ${\rm depth}(X)$ can only be $2$.
\end{proof}

Now, we are ready to give a\\

\noindent {\bf Proof of Theorem \ref{first main theorem}.}
$(\Rightarrow)$ : By Corollary \ref{cor:degree range}, it holds that
$c+2 \leq d \leq d+4$. Also if $d=c+2$, then
\begin{equation*}
a_2 (X) = {{c+1} \choose {2}} -n-2 +
\mbox{depth}(X)
\end{equation*}
(cf. Theorem \ref{thm:VAMD review}.(1)). Thus ${\rm
depth}(X)=n-1$ and $n \geq 2$.

Suppose that $d=c+3$. Then Lemma \ref{lem:hyperplane} guarantees
that the general linear curve section $\mathcal{C} \subset \P^{1+c}$ of $X$ is
either $(i)$ a smooth rational curve of $d=c+3$ having a $4$-secant
line or else $(ii)$ the image of an isomorphic projection of a
linearly normal curve of arithmetic genus one from a point. In the
former case, $\mathcal{C}$ is a curve of maximal curve and hence $X$ is a
variety of maximal sectional regularity. In the latter case,
$g(X)=1$ and ${\rm depth} (X) =n$ by Proposition \ref{prop:depth for d=c+3}.

Finally, assume that $d=c+4$. Then Proposition \ref{prop:degree
bound}.(2) says that $\mbox{depth}(X)=n+1$.
\smallskip

\noindent $(\Leftarrow)$ : For $(i)$, see again Theorem \ref{thm:VAMD review}.(1). For $(ii)$, we get ${\rm depth}(X)=n$ by Proposition \ref{prop:VMSR d=c+3}.(1). Therefore we have ${\rm depth}(X) \geq n$ for all cases $(i) \sim (iv)$. Thus the assertion comes immediately from Proposition \ref{prop:curve
case}.   \qed \\

\section{The irreducible decomposition of $\textsf{Q}(X)$ - Part I}
\noindent For an $n$-dimensional nondegenerate irreducible projective variety $X \subset \P^{n+c}$,
we denote by $\textsf{Q}(X)$ the base locus of $|\mathcal{I}_X (2)|$. The aim of this section is to prove
Theorem \ref{thm:irreducible decomposition 1} by investigating the irreducible decomposition of $\textsf{Q}(X)$
when $X$ satisfies the condition $a_2 (X) = {{c+1} \choose {2}}-k$ for some $k \in \{0,1,2,3\}$.

We begin with the following Remark \ref{rmk:trivial cases}, in which
we briefly review what is already known about Theorem \ref{thm:irreducible decomposition 1}.

\begin{remark}\label{rmk:trivial cases}
Due to Theorem \ref{thm:known results 1 classical}, Theorem \ref{thm:known results 2 Park2015}
and Theorem \ref{first main theorem}, we have a complete list of $X \subset \P^{n+c}$ in Theorem \ref{thm:irreducible decomposition 1}.
Let us divide them into four classes as follows:
\begin{enumerate}
\item[$(i)$] $d=c+1+k$ and $\mbox{depth}(X)=n+1$;
\item[$(ii)$] $d=c+2$ and $\mbox{depth}(X)=n$ or $n-1$;
\item[$(iii)$] $d=c+3$, $g(X)=1$ and $\mbox{depth}(X)=n$;
\item[$(iv)$] $d=c+3$ and $X$ is a variety of maximal sectional
regularity.
\end{enumerate}
\smallskip

\noindent $(1)$ When $X$ belongs to case $(i)$, it is well known
that $I(X)$ is generated by quadrics and hence $\textsf{Q}(X)=X$.

\noindent $(2)$ When $X$ belongs to case $(iv)$, there exists an
$n$-dimensional linear space $\F (X)$ such that $X \cap \F (X)$ is a
quartic hypersurface in $\F(X)$. In this case, it is known that
\begin{equation*}
I(X)_2 = I(X \cup \F(X))_2
\end{equation*}
and $X \cup \F (X)$ is ideal theoretically cut out by quadrics. In
consequence, we have
\begin{equation*}
\textsf{Q}(X) = X \cup \F (X).
\end{equation*}
For details, see Proposition \ref{prop:VMSR d=c+3}.

\noindent $(3)$ To study the cases $(ii)$ and $(iii)$, it turns out
to be crucial to investigate $\textsf{Q}(Z)$ when a nondegenerate
projective algebraic set $Z \subset \P^r$ is a divisor of a smooth
rational normal scroll in $\P^r$.
\end{remark}

\begin{notation and remark}\label{nr:smooth rational normal scroll}
Let $0 \leq a_1 \leq \ldots \leq a_{n+1}$ be a sequence of integers where $a_n$ and $a_{n+1}$ are positive. Consider
the $(n+1)$-fold rational normal scroll
\begin{equation*}
Y:= S(a_1 , \ldots , a_{n+1} ) \subset \P^{n+c}
\end{equation*}
of numerical type $(a_1 , \ldots , a_n , a_{n+1})$. Thus
$c=a_1 + \cdots + a_{n+1}$. The divisor class group of $Y$ is freely
generated by the hyperplane section $H$ and a linear subspace $F
\subset Y$ of dimension $n$. So, we can write the divisor class of
an effective divisor $X$ of $Y$ as $aH+bF$ for some integers $a\geq
0$ and $b \in \Z$. Several basic properties of $X$ are invariant
insider its divisor class. For example, $X$ is a nondegenerate
subscheme of $\P^{n+c}$ in the sense that its homogeneous ideal
contains no nonzero linear forms if and only if
\begin{equation}\label{eq:nondegenerate}
a=0 \quad {\mbox and} \quad b \geq 1+a_{n+1} \quad {\mbox or} \quad
a=1 \quad {\mbox and} \quad b \geq 1 \quad {\mbox or} \quad a \geq 2
\quad {\mbox and} \quad b \geq -a a_{n+1}
\end{equation}
(cf. \cite[Lemma 2.2]{Pa3}).
\end{notation and remark}

\begin{proposition}\label{prop:quadratic scheme divisor of SRNS case}
Let $Y$ and $X$ be as in Notation and Remark \ref{nr:smooth rational
normal scroll}. We assume that the divisor class $aH+bF$ of $X$
satisfies (\ref{eq:nondegenerate}) and hence $X \subset \P^{n+c}$ is
nondegenerate. Then
\smallskip

\renewcommand{\descriptionlabel}[1]%
             {\hspace{\labelsep}\textrm{#1}}
\begin{description}
\setlength{\labelwidth}{13mm} \setlength{\labelsep}{1.5mm}
\setlength{\itemindent}{0mm}

\item[$(1)$] The following three statements are equivalent:
\begin{enumerate}
\item[$(i)$] $\textsf{Q} (X) = Y$;
\item[$(ii)$] $a_2 (X) = a_2 (Y)$;
\item[$(iii)$] Either $a=0$ and $b \geq 1+2a_{n+1}$ or $a=1$ and $b
\geq 1+a_{n+1}$ or $a =2$ and $b \geq 1$ or $a \geq 3$.
\end{enumerate}
\smallskip

\item[$(2)$] Let $k \geq 1$ be the integer such that $a_k < a_{k+1} = \cdots =
a_{n+1}$ and let $X_0$ denote the subvariety $S(a_1 , \ldots , a_k
)$ of $Y$. Then
\smallskip

\renewcommand{\descriptionlabel}[1]%
             {\hspace{\labelsep}\textrm{#1}}
\begin{description}
\setlength{\labelwidth}{13mm} \setlength{\labelsep}{1.5mm}
\setlength{\itemindent}{0mm}
\item[$(2.a)$] If $a=2$ and $-2a_{n+1} \leq b \leq 0$, then $\textsf{Q}
(X) = X$.
\item[$(2.b)$] If $a=1$ and $1 \leq b \leq a_1$, then
$\textsf{Q} (X) = X$.
\item[$(2.c)$] If $a=1$ and $a_1 + 1 \leq b \leq a_{n+1}$,
then $\textsf{Q} (X) \subseteq X \cup X_0$.
\end{description}
\end{description}
\end{proposition}

\begin{proof}
From the short exact sequence $0 \rightarrow \mathcal{I}_Y
\rightarrow \mathcal{I}_X  \rightarrow \mathcal{O}_Y (-X)
\rightarrow 0$, we get
\begin{equation*}
0 \rightarrow H^0 (\P^{n+c} , \mathcal{I}_Y (2)) \rightarrow H^0
(\P^{n+c} , \mathcal{I}_X (2)) \rightarrow H^0 (Y , \mathcal{O}_Y
(2H-X)) \rightarrow 0
\end{equation*}
since $Y \subset \P^{n+c}$ is projectively normal. This proves the
equality
\begin{equation}\label{eq:difference}
a_2 (X) - a_2 (Y) = h^0 \left( Y,\mathcal{O}_Y \left( (2-a)H-bF
\right) \right).
\end{equation}

\noindent $(1)$ $(i) \Leftrightarrow (ii)$ : Since $I(Y)_2 \subseteq
I(X)_2$ and $Y=\textsf{Q} (Y)$, we have
\begin{equation*}
\textsf{Q} (X) \subseteq Y.
\end{equation*}
Obviously, if $a_2 (X) > a_2 (Y)$ then $\textsf{Q} (X)
\subsetneq Y$. Thus it holds that $\textsf{Q} (X) = Y$ exactly when $a_2 (X)
= a_2 (Y)$.
\smallskip

\noindent $(ii) \Leftrightarrow (iii)$ : From (\ref{eq:difference}),
the equality $a_2 (X) = a_2 (Y)$ holds if and only if
$$h^0 \left( Y,\mathcal{O}_Y \left( (2-a)H-bF \right) \right) =0.$$
It is an elementary task to check that this occurs exactly when
$(iii)$ holds.
\smallskip

\noindent $(2)$ We prove the three cases in turn.

\noindent $(2.a)$ Suppose that $a=2$ and $-2a_{n+1} \leq b \leq 0$.
We will see that $I(X)$ is generated by $I(X)_2$ and hence
$\textsf{Q} (X) = X$. Indeed, if $-2a_{n+1} \leq b \leq -c+1$ then
$\mbox{reg}(X)=2$ and hence $I(X) = \langle I(X)_2 \rangle$. Also,
if $-c+2 \leq b \leq 0$ then $X$ is arithmetically Cohen-Macaulay
and $\mbox{deg}(X) \leq 2c$. Let $\Gamma \subset \P^c$ be a general
zero-dimensional linear section of $X$. Since $|\Gamma| \leq 2c$,
the homogeneous ideal of $\Gamma$ is generated by its quadratic
equations. This shows that $I(X) = \langle I(X)_2 \rangle$.
\smallskip

\noindent $(2.b)$ Suppose that $a=1$ and $1 \leq b \leq a_1$. Then
$X$ is ideal-theoretically cut out by quadrics and hence $\textsf{Q}
(X) = X$. For details, see \cite[Lemma 3.1 and Theorem
3.2.(2)]{Pa1}.
\smallskip

\noindent $(2.c)$ Suppose that $a=1$ and $a_1 + 1 \leq b \leq
a_{n+1}$. For each $k+1 \leq i \leq n+1$, let $X_i$ denote the
sub-scroll $S(a_1 , \ldots , \hat{a_i} , \ldots , a_{n+1})$ of $Y$.
Then
\begin{equation*}
X_i \equiv H-a_i F = H-a_{n+1} F
\end{equation*}
since $a_{k+1} = \cdots = a_{n+1}$, and hence we have
\begin{equation*}
X \cup X_i \equiv 2H+(b-a_{n+1})F.
\end{equation*}
Then
\begin{equation*}
-2a_{n+1} \leq 1-a_{n+1} \leq b-a_{n+1} \leq 0
\end{equation*}
and hence $\textsf{Q} (X \cup X_i) = X \cup X_i$ by (2.a). Therefore
\begin{equation*}
\textsf{Q} (X) \subset \bigcap_{k+1 \leq i \leq n+1} \textsf{Q} (X
\cup X_i) = X \cup \left( \bigcap_{k+1 \leq i \leq n+1} X_i \right)
= X \cup X_0 .
\end{equation*}
This completes the proof.
\end{proof}

\begin{proposition}\label{prop:quadratic scheme d=c+3 curve}
Let $\mathcal{C} \subset \P^{1+c}$, $c \geq 4$, be a nondegenerate projective
integral curve of arithmetic genus $1$ and degree $d=c+3$. Then
$\mathcal{C}$ is the only nondegenerate irreducible component of
$\textsf{Q} (\mathcal{C})$.
\end{proposition}

\begin{proof}
Since $h^0 (\mathcal{C} , \mathcal{O}_{\mathcal{C}} (1)) = d+ 1-g
(\mathcal{C}) = c+3$, it holds that
\begin{equation*}
\mathcal{C} = \pi_P (\tilde{\mathcal{C}})
\end{equation*}
where $\tilde{\mathcal{C}} \subset \P^{2+c}$ is the linearly normal embedding
of $\mathcal{C}$ by the complete linear series
$|\mathcal{O}_{\mathcal{C}} (1)|$ and $\pi_P : \tilde{\mathcal{C}}
\rightarrow \P^{1+c}$ is an isomorphic linear projection of
$\tilde{\mathcal{C}}$ from a closed point $P \in \P^{2+c}$.

Note that $\tilde{\mathcal{C}}$ is contained in a smooth rational
normal surface scroll $\tilde{S} = S(m+b-2,m)$ where $0 \leq b \leq
2$ and $b+2m=c+3$ (cf. \cite[Theorem 3.1 and Corollary 3.3]{Pa2}).
Indeed, $\tilde{S}$ is smooth since
\begin{equation*}
c \geq 4 \quad \mbox{and} \quad m \geq m+b-2 \geq  \left \lceil \frac{c-1+b}{2} \right \rceil
\geq 2.
\end{equation*}
Furthermore, $P$ is not contained in $\tilde{S}$ since
$\tilde{\mathcal{C}}$ is a bisecant divisor of $\tilde{S}$ and
$\pi_P : \tilde{\mathcal{C}} \rightarrow \mathcal{C}$ is an isomorphism. This implies
that the projected surface
\begin{equation*}
S:=\pi_P (\tilde{S}) \subset \P^{1+c}
\end{equation*}
is contained in a rational normal $3$-fold scroll $T \subset
\P^{1+c}$ (cf. \cite[Theorem 1.4]{BS2}). Also $S$ is a surface of almost minimal degree.
By Theorem \ref{thm:VAMD review}.(3), either $\mbox{dim}~\mbox{Vert}(T) \leq 0$ or else
$\mbox{dim}~\mbox{Vert}(T) = 1$ and $S$ is arithmetically Cohen-Macaulay.
In the latter case, $I(S)$ is generated by quadratic forms and hence $\textsf{Q}(S) = S$.
In the former case, $S$ is linearly equivalent to $H+2F$ as a divisor of $T$ and hence Proposition \ref{prop:quadratic scheme divisor of SRNS case}.(2.c) shows that $S$ is the
unique nondegenerate irreducible component of $\textsf{Q}(S)$.

Suppose that there is a nondegenerate irreducible component $D$ of
$\textsf{Q}(\mathcal{C})$ which is different from $\mathcal{C}$. Then $D$ must be a curve since $a_2 (D) \geq a_2 (\mathcal{C}) > {{c} \choose {2}}$. Also, since
\begin{equation*}
D \subset \textsf{Q}(\mathcal{C}) \subset \textsf{Q}(S)
\end{equation*}
and $S$ is the only nondegenerate irreducible component of $\textsf{Q}(S)$, it follows that $D$ is contained in $S$. Let $\tilde{D} \subset \tilde{S}$ be the pull-back of $D$ via $\pi_P$.

\begin{equation*}
\begin{CD}
\tilde{\mathcal{C}} \quad & \quad \subset \quad & \quad
\tilde{\mathcal{C}} \cup \tilde{D} \quad & \quad \subset \quad &
\tilde{S}
& & \\
\downarrow \quad &   & \downarrow &   & \quad \quad \downarrow \pi_P
& & \\
\mathcal{C} \quad & \quad \subset & \mathcal{C} \cup D & \subset &
\quad S \quad & \quad \subset \quad & \quad T \\
\end{CD}
\end{equation*}

\noindent Observe that
\begin{equation*}
\tilde{\mathcal{C}} \cup \tilde{D} \equiv aH+bF \quad \mbox{for
some} \quad a \geq 3.
\end{equation*}
In particular, it holds that
\begin{equation*}
I(\tilde{\mathcal{C}} \cup \tilde{D})_2 = I(\tilde{S})_2 \quad \mbox{and hence} \quad \textsf{Q}(\tilde{\mathcal{C}} \cup \tilde{D}) = \tilde{S}
\end{equation*}
by Proposition \ref{prop:quadratic scheme divisor of SRNS case}.
Then it follows that
\begin{equation*}
I(\mathcal{C})_2 = I(\mathcal{C} \cup D)_2 = I(S)_2 .
\end{equation*}
In particular, we get
\begin{equation*}
a_2 (S) = a_2 (\mathcal{C}) = {{c+1} \choose {2}} -3 = {{c} \choose {2}} + c-3 > {{c} \choose {2}} ,
\end{equation*}
which is a contradiction. In consequence, $\textsf{Q} (\mathcal{C})$ has no nondegenerate irreducible
component except $\mathcal{C}$.
\end{proof}

Now, we are ready to give a\\

\noindent {\bf Proof of Theorem \ref{thm:irreducible decomposition
1}.} In Remark \ref{rmk:trivial cases}, we divide the varieties in
consideration into four cases. When $X$ belongs to the cases $(i)$
and $(iv)$, the statement of our theorem is already shown (cf.
Remark \ref{rmk:trivial cases}). Also if $X$ belongs to the case
$(ii)$, then it is contained in an $(n+1)$-fold rational normal
scroll as a divisor. So, $X$ satisfies the desired property by
Proposition \ref{prop:quadratic scheme divisor of SRNS case}.

Now, suppose that $X$ belongs to the case $(iii)$. Namely, $X
\subset \P^{n+c}$ is a nondegenerate projective irreducible variety
of dimension $n$ and degree $d=c+3$ such that $g(X)=1$. Let
$\mathcal{C} \subset \P^{1+c}$ be a general linear curve section of
$X$. Then $a_2 (X) = a_2 (\mathcal{C})$ by Lemma
\ref{lem:hyperplane} and hence $I(\mathcal{C})_2$ is equal to the
restriction of $I(X)_2$ to $\mathcal{C}$. In particular, it holds
that
\begin{equation*}
\textsf{Q} (\mathcal{C}) = \textsf{Q} (X) \cap \P^{1+c}.
\end{equation*}
By Proposition \ref{prop:quadratic scheme d=c+3 curve}, we know that
$\mathcal{C}$ is the only nondegenerate irreducible component of
$\textsf{Q} (\mathcal{C})$. This shows that $\textsf{Q} (X)$ has no
nondegenerate irreducible component of dimension $\geq n$ except
$X$. \qed \\

\section{The irreducible decomposition of $\textsf{Q}(X)$ - Part II}
\noindent This section is devoted to studying the irreducible
decomposition of $\textsf{Q}(X)$ when $X \subset \P^{n+c}$ is an
$n$-dimensional nondegenerate irreducible projective variety of
degree $d$ such that $d \geq 2c+2k+3$ and $a_2 (X) = {{c} \choose
{2}}-k$ for some $k \in \{0,1,2,3\}$. The following our main result
in this section shows how this problem is related to the
classification theory obtained in Theorem \ref{thm:known results 1
classical}, Theorem \ref{thm:known results 2 Park2015} and Theorem
\ref{first main theorem}.

\begin{theorem}\label{thm:irreducible decomposition 2}
Let $X \subset \P^{n+c}$ be a nondegenerate projective variety of
dimension $n$ and degree $d \geq 2c+2k+3$ such that $a_2 (X) = {{c}
\choose {2}}-k$ for some $k \in \{0,1,2,3\}$. Then there exists a
unique projective variety $Y \subset \P^{n+c}$ satisfying the
following three conditions:
\begin{enumerate}
\item[$(i)$] $\mbox{dim}~Y = n+1$;
\item[$(ii)$] $X \subset Y$ $($and hence $Y$ is nondegenerate in
$\P^{n+c}$$)$;
\item[$(iii)$] $I(Y)_2 = I(X)_2$ $($and hence $\textsf{Q}(Y) =
\textsf{Q}(X)$ and $a_2 (Y) = {{c} \choose {2}}-k$$)$.
\end{enumerate}
\end{theorem}

We will give the proof of this theorem at the end of this section.
We begin with some well-known and beautiful results in Castelnuovo
Theory.
 
\begin{theorem}[G. Castelnuovo, G. Fano and Eisenbud-Harris in \cite{H}, I. Petrakiev in \cite{Pe}]\label{thm:Castelnuovo Theory}
Let $k$ be an integer in $\{ 0,1,2,3 \}$. Suppose that 
$$c \geq \begin{cases} k+2 & \mbox{if $0 \leq k \leq 2$, and} \\
                       6   & \mbox{if $k=3$.} \end{cases} $$
Let $\mathcal{C} \subset \P^{1+c}$ be a nondegenerate projective integral curve of degree $d$ and let $\Gamma \subset \P^c$ be a general hyperplane section of $\mathcal{C}$. If
$a_2 (\Gamma) = {{c} \choose {2}}-k$, then $\Gamma$ lies on a nondegenerate projective integral curve $D \subset \P^c$ of degree $\leq c+k$.
\end{theorem}

\begin{proposition}\label{prop:irreducible decomposition}
Let $X \subset \P^{n+c}$ be a nondegenerate projective variety of
dimension $n$ and degree $d$ such that $a_2 (X) = {{c} \choose
{2}}-k$ for some $0 \leq k \leq c-2$.
\smallskip

\renewcommand{\descriptionlabel}[1]%
             {\hspace{\labelsep}\textrm{#1}}
\begin{description}
\setlength{\labelwidth}{13mm} \setlength{\labelsep}{1.5mm}
\setlength{\itemindent}{0mm}
\item[$(1)$] Let $Y$ be an irreducible component of $\textsf{Q}(X)$ such that $Y
\subset \P^{n+c}$ is nondegenerate. Then $\mbox{dim}~Y \leq n+1$.

\item[$(2)$] Let $Y$ be an irreducible component of $\textsf{Q}(X)$
which contains $X$. Then either $Y=X$ or else $\mbox{dim}~Y = n+1$.
\end{description}
\end{proposition}

\begin{proof}
$(1)$ If $m := \mbox{dim}~Y > n+1$, then we have
\begin{equation*}
a_2 (Y ) \leq {{n+c-m +1} \choose {2}} < {{c-1} \choose {2}}.
\end{equation*}
This contradicts to the assumption that
\begin{equation*}
a_2 (Y ) \geq a_2 (X) = {{c-1} \choose {2}}+(c-1-k) \leq {{c-1}
\choose {2}}.
\end{equation*}
Thus it must hold that $m \leq n+1$.
\smallskip

\noindent $(2)$ The assertion comes immediately by $(1)$.
\end{proof}

Now, we are ready to give a proof of Theorem \ref{thm:irreducible
decomposition 2}.\\

\noindent {\bf Proof of Theorem \ref{thm:irreducible decomposition
2}.} First we will show that there exists an irreducible component
of $\textsf{Q}(X)$ which satisfies the conditions $(i)\sim(iii)$.
Let
\begin{equation*}
\textsf{Q}(X) = Y_1 \cup \cdots \cup Y_{\ell}
\end{equation*}
be the minimal irreducible decomposition of $\textsf{Q}(X)$ and let
$\Gamma \subset \P^c$ be a zero-dimensional general linear section
of $X$. Since $d \geq 2c+2k+3$, we have
\begin{equation*}
a_2 (X) \leq a_2 (\Gamma) \leq {{c} \choose {2}}
\end{equation*}
and hence we can write $a_2 (\Gamma) = {{c} \choose {2}}-\epsilon$
for some non-negative integer $\epsilon \leq k$. Then Theorem
\ref{thm:Castelnuovo Theory} says that $\Gamma$ lies on a curve $D$
of degree $\leq c+\epsilon$. Moreover, it holds that $a_2 (\Gamma) =
a_2 (D)$. Indeed, consider the short exact sequence
\begin{equation*}
0 \rightarrow \mathcal{I}_D \rightarrow \mathcal{I}_{\Gamma}
\rightarrow \mathcal{I}_{\Gamma/D} \rightarrow 0.
\end{equation*}
Then we have the long exact sequence
\begin{equation*}
0 \rightarrow H^0 (\P^c , \mathcal{I}_D (2)) \rightarrow H^0 (\P^c ,
\mathcal{I}_{\Gamma} (2)) \rightarrow H^0 (D ,
\mathcal{I}_{\Gamma/D} (2)) \rightarrow \cdots
\end{equation*}
and hence it suffices to show that $H^0 (D , \mathcal{I}_{D/\Gamma}
(2))=0$. We can regard $H^0 (D , \mathcal{I}_{D/\Gamma} (2))$ as a
subspace of $H^0 (D,\mathcal{O}_D (2))$. In particular, if $H^0 (D ,
\mathcal{I}_{\Gamma/D} (2))$ has a non-zero element $s$ then $V(s)$
in $D$ contains $\Gamma$ and hence
\begin{equation*}
|V(s)| \geq |\Gamma|> 2 \times \mbox{deg}(D).
\end{equation*}
This is impossible since $D$ is locally Cohen-Macaulay.
Consequently, it must hold that $H^0 (D , \mathcal{I}_{\Gamma/D}
(2))=0$. Now, it follows that $I(\Gamma)_2 = I(D)_2$ and hence
$\textsf{Q}(D) = \textsf{Q}(\Gamma)$. From
\begin{equation*}
\Gamma \subset D \subset \textsf{Q}(D) = \textsf{Q}(\Gamma) \subset
\textsf{Q}(X) \cap \P^c = \bigcup_{1 \leq i \leq \ell} (Y_i \cap
\P^c ),
\end{equation*}
we have $D \subset Y_i \cap \P^c$ for some $i$. Then $Y_i \subset
\P^{n+c}$ is nondegenerate and hence Proposition
\ref{prop:irreducible decomposition} shows that $\mbox{dim}~Y_i \leq
n+1$. Then $\mbox{dim}~Y_i = n+1$ since $D \subset Y_i \cap \P^c$.
Therefore $D=Y_i \cap \P^c$. Also note that
\begin{equation*}
\mbox{dim}~(X \cap Y_i ) = \mbox{dim}~(X \cap Y_i \cap \P^c) + n
\geq \mbox{dim}~ \Gamma + n = n.
\end{equation*}
Thus it holds that $X \subset Y_i$ and hence $I(Y_i )_2 \subseteq
I(X)_2$. Finally, we have $I(X)_2 \subseteq I(Y_i )_2$ since $Y_i
\subset \textsf{Q}(X)$. In consequence, it holds that $I(X)_2 =
I(Y_i )_2$.

To prove the uniqueness of $Y$, note that $Y_i \subset \P^{n+c}$ is
an $(n+1)$-dimensional nondegenerate projective variety with $a_2
(Y_i )={{c} \choose {2}}-k$ for some $0 \leq k \leq 3$. Thus $Y$ is
the only nondegenerate irreducible component of $\textsf{Q}(Y)$ of
dimension $\geq n+1$. This completes the proof since
$\textsf{Q}(Y)=\textsf{Q}(X)$.  \qed \\


\begin{thebibliography}{0000000}
\bibitem[BLPS1]{BLPS1} M. Brodmann, W. Lee, E. Park and P. Schenzel, {\em On projective varieties of maximal sectional regularity}, J. Pure Appl. Algebra 221 (2017), 98–118.

\bibitem[BLPS2]{BLPS2} M. Brodmann, W. Lee, E. Park and P. Schenzel, {\em On surfaces of maximal sectional regularity}, Taiwanese J. Math. 21 (2017), no. 3, 549–567.

\bibitem[BP1]{BP1} M. Brodmann and E. Park, {\em On varieties of almost minimal degree I: Secant loci of rational normal scrolls},
J. Pure Appl. Algebra 214 (2010), 2033-2043.

\bibitem[BP2]{BP2} M. Brodmann and E. Park, {\em On varieties of almost minimal degree III: Tangent spaces and embedding scrolls},
J. Pure Appl. Algebra 214 (2011), 2859-2872.

\bibitem[BS1]{BS1} M. Brodmann and P. Schenzel, {\em On projective curves of maximal regularity}, Math.
Zeit. 244 (2003), 271 - 289.

\bibitem[BS2]{BS2} M. Brodmann and P. Schenzel, {\em Arithmetic
properties of projective varieties of almost minimal degree}, J.
Algebraic Geometry 16 (2007), 347 - 400.

\bibitem[BS3]{BS3} M. Brodmann and P. Schenzel, {\em Projective curves with maximal regularity and applications to syzygies and surfaces}, Manuscripta Math. 135, (2011), 469-495.

\bibitem[C]{C} G. Castelnuovo, {\em Sui multipli di une serie lineare di gruppi di punti
appartenente ad une curva algebraic}, Rend. Circ. Mat. Palermo (2) 7 (1893), 89-110.

\bibitem[Fa]{Fa} G. Fano, {\em Sopra le curve di dato ordine e dei massimi generi in uno
spazio qualunque}, Mem Accad. Sci. Torino 44, (1894) 335-382.

\bibitem[Fu1]{Fu1} T. Fujita, {\em Projective varieties of $\delta$-genus one}. in Algebraic and Topological Theories, to the memory of T. Miyata, Kinokuniya (1985), 149 - 175.

\bibitem[Fu2]{Fu2} T. Fujita, {\em On Singular del Pezzo varieties}. in Algebraic Geometry, Proceedings of L'Aquila
Conference, 1988, Springer Lecture Notes in Math., 1417 (1990), 116
- 128.

\bibitem[Fu3]{Fu3} T. Fujita, {\em  Classification theories of polarized varieties}. London Mathematical Society Lecture Notes Series 155, Cambridge University Press, 1990.

\bibitem[GL]{GL} M. Green and R. Lazarsfeld, {\em Some results on the syzygies of finite sets and algebraic curves}, Compos. Math. 67 (1988), 301-314.

\bibitem[GLP]{GLP} L. Gruson, R. Lazarsfeld and C. Peskine,
{\em On a theorem of Castelnuovo and the equations defining
projective varieties}, Invent. Math. 72 (1983), 491-506.

\bibitem[H]{H} J. Harris (with the collaboration of D. Eisenbud), {\em Curves in projective space}, Montr\'{e}al: Les Presses de l'Universit\'{e} de Montr\'{e}al (1982).

\bibitem[HSV]{HSV} L. T. Hoa, J. St\"{u}ckrad and W. Vogel, {\em Towards a structure theory for projective varieties of degree=codimension+2}, J. Pure Appl. Algebra 71 (1991), 203-231.

\bibitem[L]{L} S. L'vovsky, {\em On inflection points, monomial curves, and hypersurfaces containing projective curves}, Math. Ann. 306 (1996), 719-735.

\bibitem[Pa1]{Pa1} E. Park, {\em Smooth varieties of almost minimal degree}, J. Algebra 314 (2007), 185-208.

\bibitem[Pa2]{Pa2} E. Park, {\em Higher syzygies of hyperelliptic curves}, J. Pure Appl. Algebra 214 (2010), 101-111.

\bibitem[Pa3]{Pa3} E. Park, {\em On hypersurfaces containing projective varieties}, Forum Math. 27 (2015), 843-875.

\bibitem[Pe]{Pe} I. Petrakiev, {\em Castelnuovo theory via Gr\"{o}bner bases}, J. reine angew. Math. 619 (2008),
49-73.
\end{thebibliography}
\end{document}